\documentclass[12pt,sort&compress,preprint]{elsarticle}

\usepackage[english]{babel}
\usepackage[T1]{fontenc}
\usepackage[utf8]{inputenc}
\usepackage{lmodern}

\usepackage{overpic}
\usepackage[outline]{contour}
\usepackage{amsmath,amsfonts,amssymb,amsthm}
\usepackage{mathtools}
\usepackage{commath}
\usepackage{esvect}

\usepackage{todonotes}
\usepackage{nomencl}
\makenomenclature

\usepackage{hyperref}

\newcommand{\qi}{\mathbf{i}}
\newcommand{\qj}{\mathbf{j}}
\newcommand{\qk}{\mathbf{k}}
\newcommand{\eps}{\varepsilon}
\newcommand{\C}{\mathbb{C}}
\renewcommand{\H}{\mathbb{H}}
\newcommand{\R}{\mathbb{R}}
\renewcommand{\DH}{\mathbb{DH}}
\newcommand{\cj}[1]{{#1}^\ast}
\newcommand{\ej}[1]{{#1}_\eps}
\newcommand{\cej}[1]{{#1}^\ast_\eps}
\newcommand{\vecLp}{L_p}
\newcommand{\vecLd}{L_d}
\newcommand{\cjvecLp}{\cj{L}_p}
\newcommand{\cjvecLd}{\cj{L}_d}
\newcommand{\vecKp}{K_p}
\newcommand{\vecKd}{K_d}

\DeclareMathOperator{\rgcd}{rgcd}

\newtheorem{theorem}{Theorem}
\newtheorem{lemma}{Lemma}
\theoremstyle{definition}
\newtheorem{definition}{Definition}
\theoremstyle{remark}
\newtheorem{remark}{Remark}
\newtheorem{example}{Example}

\begin{document}

\begin{frontmatter}
\journal{}

\title{Rational Motions of Minimal Quaternionic Degree with Prescribed Line Trajectories}

\author{Zülal Derin Yaqub\corref{cor}}
\ead{derinzulal@gmail.com}
\author{Hans-Peter Schröcker}
\ead{hans-peter.schroecker@uibk.ac.at}
\address{University of Innsbruck, Department of Basic Sciences in Engineering Sciences, Technikerstr.~13, 6020 Innsbruck, Austria}
\cortext[cor]{Corresponding author}

\begin{abstract}
  In this paper, we study how to find rational motions that move a line along a
  given rational ruled surface. Our goal is to find motions with the lowest
  possible degree using dual quaternions. While similar problems for point
  trajectories are well known, the case of line trajectories is more complicated
  and has not been studied. We explain when such motions exist and how to
  compute them. Our method gives explicit formulas for constructing these
  motions and shows that, in many cases, the solution is unique. We also show
  examples and explain how to use these results to design simple mechanisms that
  move a line in the desired way. This work helps to better understand the
  relationship between rational motions and ruled surfaces and may be useful for
  future research in mechanism design.
\end{abstract}


\end{frontmatter}

\begin{thenomenclature}
\nomgroup{Symbols}
  \item [{$\H$}]\begingroup algebra of quaternions\nomeqref {0}\nompageref{1}
  \item [{$\DH$}]\begingroup algebra of dual quaternions\nomeqref {0}\nompageref{1}
  \item [{$C = P + \eps D$}]\begingroup motion polynomial with primal part $P$ and dual   part $D$\nomeqref {0}\nompageref{1}
  \item [{$L = L_p + \eps L_d$}]\begingroup rational ruled surface with primal part   $L_p$ and dual part $L_d$\nomeqref {0}\nompageref{1}
  \item [{$\cj{c}$, $\cj{h}$, $\cj{C}$}]\begingroup conjugate quaternion, dual   quaternion, or motion polynomial\nomeqref {0}\nompageref{1}
  \item [{$\ej{c}$, $\ej{h}$, $\ej{C}$}]\begingroup $\eps$-conjugate quaternion, dual   quaternion, or motion polynomial\nomeqref {0}\nompageref{1}
  \item [{$\rgcd C$}]\begingroup greatest common real divisior of quaternionic   polynomial\nomeqref {0}\nompageref{1}

\end{thenomenclature}

\section{Introduction}
\label{sec:introduction}

In the article \cite{selig24} the authors write about the ``design of mechanisms
that move a line along a desired trajectory.'' They construct some particular
mechanisms to generate ruled surfaces and also line congruences. This topic can
be regarded as extension of the path generation problem -- constructing a
mechanism to a given point path -- to line space. While point path generation is
a standard task in mechanism science, it is still considered difficult and often
is solved via a de-tour to motion generation \cite[Section~1.4]{mccarthy11}.
Therefore, it is not surprising that the authors of \cite{selig24} continue
``Indeed the authors could find no mention of this problem in the literature of
the past 120 years or so.''

We agree with this assessment: The mechanical generation of ruled surfaces is
not a well-explored topic. This is maybe a bit surprising as ruled surfaces play
an important role in space kinematics. Examples include the kinematic generation
of ruled surfaces \cite{sprott02} but also ruled surfaces related to important
mechanisms \cite{perez02} and motions \cite[Chapter~IX]{Bottema90}. These
references are far from being exhaustive.

In this article, we build on ideas to construct mechanisms with prescribed
\emph{rational} point trajectories that culminated in a version of Kempe's
Universality Theorem for rational curves \cite{hegedus15,Gallet16,li18}. In
contrast to more general results on algebraic curve generation, it provides a an
asymptotic bound on the number of links and joints that is linear in the curve
degree and also more practical mechanisms to low-degree curves are rather
simple. The first step of the construction in \cite{li18} determines a rational
motion of minimal degree with a prescribed point trajectory \cite{li16} which
then is turned into a mechanism using the factorization theory of rational
motions \cite{hegedus13}. At least in generic cases, this last step is supported
by a mature theory. An open source reference implementation in Python exists
\cite{huczala2024}.

In this article, we pursue a similar target but we not yet aim at a version of
Kempe's universality theorem for ruled surfaces. Instead we focus on the first
step, the construction of rational motions with a prescribed line trajectory.
Arguably, this is more interesting than the corresponding problem for point
trajectories where one can always take the rather trivial translation along the
curve. It is known that this translational motion is generically of minimal
degree in the dual quaternion model of space kinematics. Only in special cases
-- for ``circular'' curves -- rational motions of smaller degree exist
\cite{li16}. Rather surprisingly, the minimal degree rational motion with
prescribed rational point trajectory is \emph{always unique.}

In the following we will consider existence, uniqueness, and computation of
rational motions of minimal degree with a \emph{prescribed line trajectory.} Our
main results are quite similar to the case of rational plane trajectories
\cite{derin25} but differ considerably from the curve case:
\begin{itemize}
\item Rational motions with a given line trajectory exist if and only if a
  certain (rather obvious) condition on the ruled surface is satisfied: Its
  spherical images needs to be rational.
\item In general, the rational motion of minimal degree is unique but -- in
  contrast to the curve case -- there are also examples with non-unique rational
  motions of minimal degree.
\item The spherical component of the minimal degree motion can be computed by
  known algorithms that originate in research on Pythagorean hodograph curves
  \cite{choi02,schroecker24}. We succeed in providing surprisingly compact
  expressions for the translational component: Closed algebraic formulas that
  involve the spherical motion component and expressions obtained by polynomial
  division with remainder and the extended Euclidean algorithm.
\end{itemize}

We continue this article by introducing basic definitions, concepts, and
notation in Section~\ref{sec:preliminaries} before considering existence of
rational motions with prescribed rational line trajectory (rational ruled
surfaces) in Section~\ref{sec:existence}. Our proof of the main existence result
(Theorem~\ref{th:1}) is constructive and provides more or less explicit formulas
for computing the translational motion component. In
Section~\ref{sec:minimal-degree} we demonstrate that the motions obtained in
Section~\ref{sec:existence} contain motions of minimal degree and we show how to
single them out computationally.

We then connect our results to that of \cite{selig24} by discussing the minimal
degree motion of the cylindroid (Plücker conoid) in Section~\ref{sec:cylindroid}
and, in Section~\ref{sec:mechanism-design}, we discuss the construction of a
simple mechanism from a rational motion with prescribed line trajectory at hand
of an example.

As already mentioned, our results and also proofs for the case of line
trajectories are quite similar to the case of plane trajectories, which was
treated in \cite{derin25}, but very different from point trajectories. In order
to avoid unnecessary duplication, we will sometimes use results of
\cite{derin25} or refer the reader to that paper for a more detailed discussion.
However, we will highlight important differences and provide all necessary
information to implement our formulas and algorithms. Furthermore, we illustrate
theoretical results in numerous examples.

\section{Preliminaries}
\label{sec:preliminaries}

This section establishes the foundational definitions and conditions relevant to
our study.

\subsection{Dual Quaternions and Motion Polynomials}

Let \( \H \) denote the skew field of quaternions and \( \DH \) the algebra of
dual quaternions. An element \( q \in \DH \) has the form:
\[
  c = p + \eps d = (p_0 + p_1 \qi + p_2 \qj + p_3 \qk) + \eps (d_0 + d_1 \qi + d_2 \qj + d_3 \qk),
\]
where $\eps$ is an algebra element that squares to zero ($\eps^2 = 0$) and
commutes with the quaternion units $\qi$, $\qj$, $\qk$.

The \emph{conjugate} of $c$ is $\cj{c} = \cj{p} + \eps \cj{d}$, with $\cj{p} =
p_0 - p_1 \qi - p_2 \qj - p_3 \qk$ and $\cj{d} = d_0 - d_1\qi - d_2\qj -
d_3\qk$. The conjugate of a product of dual quaternions equals the product of
the conjugate factors in reverse order, $\cj{(ab)} = \cj{b}\cj{a}$. We will also
use the notion of the \emph{$\eps$-conjugate} $\ej{c} \coloneqq p - \eps d$. The
dual quaternion $c$ is called \emph{vectorial} if $c - \cj{c} = 0$, that is $p_0
= d_0 = 0$. The norm of $c$ is the quantity $c\cj{c} = p\cj{p} + \eps(p\cj{d} +
d\cj{p})$. Since both, $p\cj{p}$ and $p\cj{d} + d\cj{p}$, are real, the norm is
a dual number in general. Note that we conveniently identify the vector space of
vectorial quaternions $p = p_1\qi + p_2\qj + p_3\qk$ with $\R^3$.

By $\DH[t]$ we denote the ring of polynomials in the indeterminate $t$ and with
coefficients in the dual quaternions. We will use it to parametrize rational
rigid body motions, as in \cite{hegedus13}, whence the indeterminate $t$ serves
as a real motion parameter and thus commutes with all coefficients. A polynomial
in $\DH[t]$ can be written in a unique way as $C(t) = P(t) + \eps D(t)$ where
both, the primal part $P(t)$ and the dual part $D(t)$, are in the ring $\H[t]$
of quaternionic polynomials. Note that we will usually drop the argument $t$,
that is, we just write $C = P + \eps D$.

\begin{definition}
  \label{def:motion-polynomial}
  The polynomial \( C = P + \varepsilon D \in \mathbb{DH}[t] \) is called a
  \emph{motion polynomial} if its norm polynomial $C\cj{C}$ satisfies
  \[
    C\cj{C} = P \cj{P} + \eps (P \cj{D} + D \cj{P}) \in \R[t] \setminus \{0\},
  \]
  that is,
  \[
    P\cj{D} + D\cj{P} = 0 \quad\text{and}\quad P \neq 0.
  \]
  The first condition in this equation is called the \emph{Study condition.}
\end{definition}

Motion polynomials parametrize \emph{rational rigid body motions} by a suitable
action on points, but also by an induced action on lines and planes. The
trajectory of the point $(x,y,z)$ is, for example, given by
\begin{equation}
  \label{eq:1}
  1 + \eps(x\qi + y\qj + z\qk) \mapsto
  \frac{C_\eps (1 + \eps(x\qi + y\qj + z\qk)) \cj{C}}
       {C\cj{C}}.
\end{equation}
Rational motions are characterized by having rational point trajectories
(rational curves) but they also have rational line trajectories (rational ruled
surfaces).

\subsection{Plücker Coordinates and Rational Ruled Surfaces}

A line in space is represented by its \emph{Plücker vector.} We view it as a
vectorial dual quaternion
\[
  L = \vecLp + \eps \vecLd,
\]
where $\vecLp \in \R^3$ is a direction vector and $\vecLd \in \R^3$ is the
corresponding moment vector. The Plücker condition
\[
  \vecLp \cjvecLd + \vecLd \cjvecLp = 0
\]
ensures geometric validity. Note that we don't generally require $\vecLp$ to be
normalized; Plücker coordinates in this article are homogeneous.

\begin{definition}
  A polynomial $L = \vecLp + \eps\vecLd \in \DH[t]$ is called a \emph{line
    polynomial} if it satisfies
\[
  L - \cj{L} = 0,\quad
  \vecLp \cjvecLd +\vecLd \cjvecLp = 0,\quad
  \vecLp \neq 0.
\]
\end{definition}

Line polynomials are special motion polynomials and describe line symmetric
motions \cite{hamann11}. Here, we rather view a line polynomial as a
parameterization of a ruled surface in Plücker coordinates. The vectorial
quaternionic polynomials $\vecLp$, $\vecLd \in \H[t]$ represent the directions
and moments of the rulings, respectively. The defining conditions ensure that
$L$ is a valid Plücker vector for almost all values of $t \in \R$. At isolated
instances it is possible that the primal part vanishes. If the dual part is
non-zero, the corresponding line is at infinity. Since Plücker coordinates are
homogeneous, the ruled surface itself is a \emph{rational} ruled surface.

One way of generating a rational ruled surface is by acting on a fixed line in
the moving space, say $\qk$, with a motion polynomial $C = P + \eps D \in
\DH[t]$. This action is given by the formula
\begin{equation}
  \label{eq:2}
  L = \ej{C} \qk \cej{C},
\end{equation}
cf. \cite{radavelli14}.\footnote{Note that the formula in \cite{radavelli14}
  comes without $\eps$-conjugation. The reasons for this is a slightly different
  convention for the action \eqref{eq:1} of dual quaternions on points. Formulas
  of this paper can be adapted to the convention of \cite{radavelli14} by
  changing the sign of the dual part~$D$.} However, not every rational ruled
surface can be generated in this way. Re-writing \eqref{eq:2} in terms of primal
and dual parts, we obtain
\[
  L = \vecLp + \eps\vecLd = P\qk\cj{P} - \eps(P\qk\cj{D} + D\qk\cj{P}).
\]
Observing that $P\cj{P} = \cj{P}P$ is real and hence commutes with everything,
we see that the norm polynomial equals
\[
  L\cj{L} = \vecLp\cjvecLp = (P\qk\cj{P})\cj{(P\qk\cj{P})} =
  P\qk\cj{P}P\qk\cj{P} = (P\cj{P})^2.
\]
It is thus the \emph{square} of the norm of the primal part of $C$. This
motivates the next definition:

\begin{definition}
  We call the rational ruled surface $L = \vecLp + \eps\vecLd$ \emph{kinematic,}
  if its norm $L\cj{L} = \vecLp\cjvecLp$ is a square in~$\R[t]$.
\end{definition}

\begin{example}
  One family of rulings of a hyperbolic of revolution constitutes a kinematic
  rational ruled surfaces. It is obtained by acting on a fixed line with a
  generic linear motion polynomial $C = c_1t + c_0 \in \DH[t]$ which describes
  the composition of all rotations around a fixed axis with a fixed rigid body
  displacement in the moving or the fixed frame. A family of rulings on a
  general hyperboloid is a rational ruled surface but not kinematic.
\end{example}

\subsection{Saturated Kinematic Rational Ruled Surfaces}

Only kinematic rational ruled surfaces can possibly arise as line trajectories
under rational motions. This necessary condition is almost sufficient -- an
observation that is actually well-known for the respective primal parts. The
precise statement is

\begin{lemma}
  \label{lem:1}
  Given a vectorial polynomial $\vecLp \in \H[t]$ whose norm is a square,
  there exists a polynomial $\ell \in \R[t]$ with only real roots of
  multiplicity one and $P \in \H[t]$ such that $\vecLp\ell = P\qk\cj{P}$.
\end{lemma}

A formal proof of this statement can be found in \cite[Lemma~2]{derin25}. Here
we just illustrate the basic ideas at hand of an example.

\begin{example}
  We consider the vectorial quaternionic polynomial
  \[
    ((-2t^3+8t^2+8t)\qi + (2t^3+8t^2-8t)\qj + (t^4+2t^2-8)\qk)(t^2+1)(t-1)^2(t-2).
  \]
  Its norm
  \begin{equation*}
\vecLp\cj{\vecLp} = (t^2+4)^2(t^2+2)^2(t^2+1)^2(t-1)^4(t-2)^2
  \end{equation*}
  is indeed a square. We denote by $\rgcd(\vecLp)$ the greatest real common
  divisor of its quaternionic coefficients $\qi$, $\qj$, $\qk$. In this example
  it equals $g \coloneqq (t^2+1)(t-1)^2(t-2)$. Consider now the vectorial
  polynomial
\[
    \vecLp/g = -2t(t^2-4t-4)\qi + 2t(t^2+4t-4)\qj + (t^2+4)(t^2-2)\qk.
  \]
  at first. It is well-known that $\rgcd(\vecLp/g) = 1$ implies existence of
  $Q \in \H[t]$ such that $\vecLp/g = Q\qk\cj{Q}$. Algorithms to compute $Q$
  are given in \cite[Theorem~4.2]{choi02} or \cite[Lemma~2.3]{schroecker24}. We
  find
  \[
    Q = t^2 -(t + 2)\qi - (t - 2)\qj - 2t \qk
  \]
  and immediately verify $Q\qk\cj{Q} = \vecLp/g$. Setting $P \coloneqq
  Q(t-\qk)(t-1)$ we now obtain $Pk\cj{P} = \vecLp/(t-2)$ which is almost what
  we want. But there is no way to generate the missing linear factor $t-2$.
  Therefore, we set $\ell = t-2$ and $P \coloneqq Q(t-\qk)(t-1)(t-2)$ so that
  \[
    \vecLp\ell = P\qk\cj{P}.
  \]
\end{example}

This example demonstrates that real linear factors of $\vecLp$ whose
multiplicity is odd cannot be generated by any $P \in \H[t]$. There is, however,
a unique monic polynomial $\ell$ of minimal degree such that $\vecLp\ell =
P\qk\cj{P}$ for some $P \in \H[t]$. We therefore define

\begin{definition}
  Given a vectorial polynomial $\vecLp$, we denote the unique monic polynomial
  $\ell$ of minimal degree such that $\vecLp\ell = P\qk\cj{P}$ for some $P \in
  \H[t]$ the \emph{minimal saturating factor} of $\vecLp$. Moreover, we say that
  the rational ruled surface $L = \vecLp + \eps\vecLd$ is \emph{saturated} if
  $\ell = 1$.
\end{definition}

Note that $\ell$ is a factor of $\rgcd(P)$, that is, $P$ can be written as $P =
Q\ell$ for some $Q \in \H[t]$.

\section{Existence of Polynomial Solutions}
\label{sec:existence}

Given a kinematic rational ruled surface \(L = \vecLp + \eps \vecLd\) we
want to find a rational motion $C = P + \eps D \in \DH[t]$ such that $L$ is the
trajectory of a fixed line in the moving frame. Assuming that this moving line
is $\qk$, this amounts to solving the equation
\begin{equation}
  \label{eq:3}
  \ej{C} \qk \cej{C} = hL
\end{equation}
where \(h \in \R[t]\) is some real polynomial that accounts for us using
homogeneous coordinates. As we already saw previously, the equation $\ej{C} \qk
\cej{C} = L$ might have no solutions so that $h \neq 1$ might be necessary to
allow for solutions. One instance of this is $L$ not being saturated but this is
not the only case that requires a co-factor $h$ of positive degree.

By comparing primal and dual parts and considering the Study condition we
convert \eqref{eq:3} in the equivalent system of equations
\begin{align}
  P \qk \cj{P} &= h\vecLp, \label{eq:4}\\
  -P \qk \cj{D} - D \qk \cj{P} &= h\vecLd, \label{eq:5}\\
  P \cj{D} + D \cj{P} &= 0 \label{eq:6}
\end{align}
that is to be solved for $P$ and $D$. The system of
Equations~\eqref{eq:4}--\eqref{eq:6} is quite similar to
\cite[Equations~(10)--(12)]{derin25} but there are also some important
differences:
\begin{itemize}
\item Equation~\eqref{eq:4} is just the same as Equation~(10) of \cite{derin25}
  with \(\vec{u}\) replaced by \(\vecLp\). We can use the well-known
  procedures \cite{choi02,schroecker24} for its solution. Solutions are unique
  up to right-multiplication with quaternions $q_0 + q_3\qk$ of unit norm, that
  is, fixed rotations around $\qk$ in the moving frame.
\item The left-hand side of \eqref{eq:5} differs from Equation~(11) of
  \cite{derin25} by one minus sign. As a consequence, it is vectorial while
  \cite[Equation~(11)]{derin25} is a scalar equation. In particular,
  \eqref{eq:5} imposes more constraints than the corresponding equation in
  \cite{derin25}. It is thus to be expected that finding solutions is more
  challenging.
\item In contrast to \cite{derin25}, the Study condition \eqref{eq:6} is not
  needed here as it is implied by the Plücker condition satisfied by $L$. We
  prove this formally in the next lemma. It will allow us to henceforth ignore
  Equation~\eqref{eq:6}.
\end{itemize}

\begin{lemma}
  \label{lem:2}
  Equations~\eqref{eq:4} and \eqref{eq:5} together with the Plücker condition
  \(\vecLp\cjvecLd + \vecLd\cjvecLp = 0\) imply~\eqref{eq:6}.
\end{lemma}

\begin{proof}
  We compute
  \[
    L = \ej{C}\qk\cej{C} = (P - \eps D) \qk (\cj{P} - \eps\cj{D})
    = P\qk\cj{P} - \eps(P\qk\cj{D}+D\qk\cj{P}).
  \]
  With
  \begin{gather*}
    \vecLp = l_1\qi + l_2\qj + l_3\qk,\quad
    \vecLd = l_5\qi + l_6\qj + l_7\qk,\quad
    P = p_0 + p_1\qi + p_2\qj + p_3\qk\\
    \text{and}\quad
    D = d_0 + d_1\qi + d_2\qj + d_3\qk
  \end{gather*}
  we obtain
  \begin{equation}
    \label{eq:7}
    l_1 = 2(p_0p_2+p_1p_3),\quad
    l_2 = 2(-p_0p1+p_2p_3),\quad
    l_3 = p_0^2-p_1^2-p_2^2+p_3^2
  \end{equation}
  as well as
  \begin{align}
    l_5 &= -2(d_0p_2+d_1p_3+d_2p_0+d_3p_1), \label{eq:8}\\
    l_6 &= -2(-d_0p_1-d_1p_0+d_2p_3+d_3p_2), \label{eq:9}\\
    l_7 &= -2(d_0p_0-d_1p_1+d_2p_2+d_3p_3). \label{eq:10}
  \end{align}
  This implies
  \begin{equation}
    \label{eq:11}
    l_1l_5 + l_2l_6 + l_3l_7 = -4(p_0^2+p_1^2+p_2^2+p_3^2)(d_0p_0 + d_1p_1 + d_2p_2 + d_3p_3).
  \end{equation}
  The left-hand side is the Plücker condition and vanishes. Because of \(P \neq
  0\), the last factor on the right-hand side vanishes as well and this
  implies~\eqref{eq:6}.
\end{proof}

We now turn to a solution of the remaining Equations~\eqref{eq:4} and
\eqref{eq:5}. In doing so, we always start with a solution $P = Qh\ell$ of
\eqref{eq:4} for the primal part, which exists by \cite[Lemma~2]{derin25}. We

For given $P = Qh\ell$ we need to solve \eqref{eq:5} for the dual part $D$. For
fixed degree $\deg D$, this amounts to solving a system of linear equations but
this viewpoint, although convenient for quick calculations, does not provide
insight on existence, uniqueness, or minimality of solutions.

The next sequence of lemmas will result in our main existence result,
Theorem~\ref{th:1} below. Our proofs are constructive and yield formulas for
computing the dual part $D$ and hence also the motion polynomial $C = P + \eps
D$. Later, we will prove that also the solutions of minimal degree are obtained
in this way. While the structure of our proofs of \cite[Lemma~3]{derin25} and of
the following sequence of lemmas are quite similar, there are important
also differences so that a detailed description is justified.

\begin{lemma}
  \label{lem:3}
  Given a solution $P = Q h\ell$ of \eqref{eq:4}, all solutions $D = d_0 +
  d_1\qi + d_2\qj + d_3\qk$ of \eqref{eq:5} \emph{over the field of rational
    functions} are given by
  \begin{align}
    d_1 &= -\frac{2d_0q_2}{2q_3}
          +\frac{(k_5(q_2^2-q_3^2)+k_6(q_0q_3-q_1q_2))}{2q_3m_3},\label{eq:12} \\
    d_2 &= \phantom{-}\frac{2d_0q_1}{2q_3}
          +\frac{(k_5(-q_0q_3-q_1q_2)+k_6(q_1^2-q_3^2))}{2q_3m_3},\label{eq:13}\\
    d_3 &= -\frac{2d_0q_0}{2q_3} +\frac{(k_5(q_1q_3+q_0q_2)+k_6(q_2q_3-q_0q_1))}{2q_3m_3}.\label{eq:14}
  \end{align}
  where $Q = q_0 + q_1\qi + q_2\qj + q_3\qk$, $M = m_1\qi + m_2\qj + m_3\qk =
  Q\qk\cj{Q}$, $k_5\qi + k_6\qj + k_7\qk = \vecKd \coloneqq \vecLd / \ell$ and
  $d_0$ is an arbitrary rational function.\footnote{Validity of
    Equations~\eqref{eq:12}--\eqref{eq:14} requires $m_3 \neq 0$ and $q_3 \neq
    0$ but both of these assumptions can be made without loss of generality, cf.
    Remark~\ref{rem:1}.}
\end{lemma}

\begin{proof}
  The proof consists of a straightforward calculation plus some simplifications
  using the Study/Plücker condition. We substitute $P = Qh\ell$ and \(\vecLd =
  \ell\vecKd\) into \eqref{eq:5} to obtain
  \begin{equation}
    \label{eq:15}
    -Q \qk \cj{D} - D \qk \cj{Q} = \vecKd.
  \end{equation}
  A direct computation provides the following solution of Equation~\eqref{eq:15}
  over the field of rational functions:
  \begin{align}
    d_1 &= -\frac{1}{{2}{q_3}{\chi}} (2d_0q_2\chi +k_5(q_2^2+q_3^2)-k_6(q_0q_3+q_1q_2)+k_7(q_0q_2-q_1q_3)),\label{eq:16} \\
    d_2 &= \frac{1}{2q_3\chi} (2d_0q_1\chi -k_6(q_1^2+q_3^2)+k_5(-q_0q_3+q_1q_2)+k_7(q_0q_1+q_2q_3)),\label{eq:17}\\
    d_3 &= -\frac{1}{2q_3\chi} (2d_0q_0\chi +k_7(q_0^2+q_3^2)+k_5(q_0q_2+q_1q_3)+k_6(q_2q_3-q_0q_1))\label{eq:18}
  \end{align}
  where $\chi \coloneqq q_0^2+q_1^2+q_2^2+q_3^2$. By the Plücker condition on
  $K$ we observe
  \begin{equation}
    \label{eq:19}
    m_1k_5 + m_2k_6 + m_3k_7 = 0.
  \end{equation}
  We substitute \(k_7 = -(m_1k_5 + m_2k_6)/m_3 \in \R[t]\) into
  \eqref{eq:16}--\eqref{eq:18} and, using
  \begin{equation}
    \label{eq:20}
    m_1 = 2(q_0 q_2 + q_1 q_3),\
    m_2 = 2(-q_0 q_1 + q_2 q_3),\
    m_3 = q_0^2  - q_1^2  - q_2^2  + q_3^2,
  \end{equation}
  we rearrange to obtain Equations~\eqref{eq:12}, \eqref{eq:13}, and
  \eqref{eq:14}.
\end{proof}

\begin{remark}
  \label{rem:1}
  Equations~\eqref{eq:12}--\eqref{eq:14} provide the most general
  \emph{rational} solution to \eqref{eq:5} provided that $m_3 \neq 0$ and $q_3
  \neq 0$. Should these assumptions be violated, we can perform a change of
  coordinates. For example, a half-turn in the moving frame $(1,0,1)$ will
  interchange $m_1$ and $m_3$ as well as $q_1$ and $q_3$. Should $q_1 = q_2 =
  q_3 = 0$, a random rotation in the moving frame will do. We will actually use
  this technique later in Section~\ref{sec:cylindroid}. Of course, there also
  exist formulas similar to \eqref{eq:12}--\eqref{eq:14} and
  \eqref{eq:16}--\eqref{eq:18} that assume the non-vanishing of other
  coefficients of $M$ and~$Q$.
\end{remark}

\begin{remark}
  \label{rem:2}
  In the proof of Lemma~\ref{lem:3} we substitute a rational expression for
  $k_7$ but -- because $k_7$ is polynomial -- the common denominator of the
  expressions in \eqref{eq:12}, \eqref{eq:13} and \eqref{eq:14} in reduced form
  is at most $q_3$, not $q_3m_3$. This statement is easy to verify for $d_3$.
  Using \eqref{eq:19} and \eqref{eq:20} we can immediately simplify
  \eqref{eq:14} to
  \begin{equation}
    \label{eq:21}
    d_3 = -\frac{2d_0q_0 + \frac{1}{2}k_7}{2q_3}.
  \end{equation}
\end{remark}

The next lemma provides us with a \emph{polynomial} solution to
Equation~\eqref{eq:5}.

\begin{lemma}
  \label{lem:4}
  If $\gcd(q_0,q_3) = 1$ then, by Bézout's identity, there exist polynomials
  $a$, $b \in \R[t]$ such that $aq_0 + bq_3 = 1$. In this case, the solution
  \eqref{eq:12}--\eqref{eq:14} to Equation~\eqref{eq:5} is polynomial if
  \begin{equation}
    \label{eq:22}
    d_0 = -\frac{1}{4}k_7a = -\frac{k_7(1-bq_3)}{4q_0}.
  \end{equation}
\end{lemma}

\begin{remark}
  The assumption $\gcd(q_0,q_3) = 1$ can again be made without loss of
  generality. The polynomials $a$ and $b$ are easy to compute by means of the
  extended Euclidean algorithm and their respective degrees are bounded by the
  degree of~$Q$. Equation~\eqref{eq:22} is motivated by \eqref{eq:21}: In order
  for $d_3$ to be polynomial, the numerator of \eqref{eq:21} should be divisible
  by $q_3$ that is,
  \begin{equation*}
    2d_0q_0 + \frac{1}{2}k_7 \equiv 0 \mod q_3.
  \end{equation*}
  The solution of this equation in the factor ring of $\R[t]/(q_3)$ is given by
  \begin{equation*}
    d_0 = -\frac{1}{4}k_7q_0^{-1} = -\frac{1}{4}k_7a \mod q_3.
  \end{equation*}
\end{remark}

\begin{proof}[Proof of Lemma~\ref{lem:4}]
  Plugging \eqref{eq:22} into \eqref{eq:21}, we find
  \begin{equation}
    \label{eq:23}
    d_3 = -\frac{-k_7(1-bq_3)+k_7}{4q_3}
    = -\frac{bk_7}{4}
  \end{equation}
  which is, indeed, a polynomial.

  We claim that \(d_0\) as in \eqref{eq:22} ensures that also the right-hand
  sides of Equations~\eqref{eq:12} and \eqref{eq:13} are polynomial. To see
  this, we re-write $d_1 = - \frac{1}{2q_3}D_1$ where
  \begin{equation*}
    D_1 = (-2d_0q_2m_3+ k_5(q_2^2-q_3^2)+k_6(q_0q_3-q_1q_2))m_3^{-1}.
  \end{equation*}
  As already argued in Remark~\ref{rem:2}, \(D_1\) is a polynomial in spite of
  the factor \(m_3^{-1}\). We need to show that $q_3$ is a factor of \(D_1\). In
  order to do so, we substitute \eqref{eq:22} to obtain
  \begin{equation*}
    D_1 = (\tfrac{1}{2}k_7aq_2m_3 - k_5(q_2^2-q_3^2)-k_6(q_0q_3-q_1q_2))m_3^{-1}.
  \end{equation*}
  Invoking \eqref{eq:19}, \eqref{eq:20}, and \(aq_0 + bq_3 = 1\) this becomes
  \begin{align}
    D_1 &= (-\tfrac{1}{2}(m_1k_5 + m_2k_6)aq_2 - k_5(q_2^2-q_2^2)-k_6(q_0q_3-q_1q_2))m_3^{-1} \notag\\
        &= \bigl(k_5((-aq_0+1)q_2^2-(aq_1q_2+q_3)q_3) \notag\\
        &\qquad\qquad\qquad + k_6((aq_0-1)q_1q_2-(aq_2^2-q_0)q_3)\bigr)m_3^{-1} \notag\\
        &= (k_5(bq_2^2-aq_1q_2-q_3) + k_6(-bq_1q_2-aq_2^2+q_0))q_3m_3^{-1}\label{eq:24}
  \end{align}
  and it transpires that, indeed, $q_3$ is a factor of $D_1$ and $d_1$ is
  polynomial. A similar computation for $d_2 = \frac{1}{2q_3}D_2$ results in
  \begin{equation}
    \label{eq:25}
    D_2 = ( k_5(-bq_1q_2+aq_1^2-q_0) + k_6(bq_1^2+aq_1q_2-q_3) )q_3m_3^{-1},
  \end{equation}
  whence $d_2$ is polynomial as well.
\end{proof}

We have shown that Equations~\eqref{eq:4}--\eqref{eq:6} admit polynomial
solutions, given by \eqref{eq:12}--\eqref{eq:14} together with \eqref{eq:22}.
Instead of Equation~\eqref{eq:14} we may also use the simpler
Equation~\eqref{eq:21} or even \eqref{eq:23}. However, using these equations, we
will in general obtain a polynomial $D$ with $\deg D > \deg P$. The next lemma
shows that, with a bit more effort, we can achieve $\deg D \le \deg P$. This is
possible because \eqref{eq:12}--\eqref{eq:14} do not describe all polynomial
solutions.

\begin{lemma}
  \label{lem:5}
  If $D = d_0 + d_1\qi + d_2\qj + d_3\qk$ is a polynomial solution to
  \eqref{eq:5}, then so is $D + \lambda Q\qk$ for any $\lambda \in \R[t]$. If
  the original solution $D$ is obtained via \eqref{eq:12}--\eqref{eq:14} together
  with \eqref{eq:22} and
  \begin{equation}
    \label{eq:26}
    \lambda = \frac{d_0 - \varrho}{q_3}
  \end{equation}
  for $\varrho$, $\lambda \in \R[t]$ and $\deg \varrho < \deg q_3$ ($\lambda$
  and $\varrho$ exist and can be computed by polynomial division with remainder)
  then $\deg (D + \lambda Q\qk) \le \deg P$.
\end{lemma}

\begin{remark}
  \label{rem:3}
  If $\deg Q < \deg P$, the polynomial $\lambda$ in Lemma~\ref{lem:5} is not
  unique. Its purpose is to get rid of high powers of $t$ in $D$. Thus, low
  degree coefficients of $\lambda$ are irrelevant and $\lambda$ can be replaced
  by $\lambda + \nu$ where $\nu \in \R[t]$ is an arbitrary polynomial of degree
  $\deg \nu < \deg P - \deg Q$. Even if $\deg Q = \deg P$, we obtain a one
  parameteric set $\{D + \lambda_0 Q\qk \mid \lambda_0 \in \R\}$ of solution
  polynomials. This is to be expected and corresponds to a translation of the
  moving frame in direction of~$\qk$.
\end{remark}

\begin{proof}[Proof of Lemma~\ref{lem:5}]
  The first statement follows from the observation that $D = Q\qk = -q_3 +
  q_2\qi - q_1\qj + q_0\qk$ is a solution of $-Q\qk\cj{D} - D\qk\cj{Q} = 0$
  which is the homogeneous equation to \eqref{eq:15}. As to the second
  statement, we compute the individual quaternion coefficients of $D + \lambda
  Q\qk$:
  \begin{itemize}
  \item From \eqref{eq:22} we find
    \begin{equation*}
      d_0 - \lambda q_3 = -\frac{1}{4}k_7a -
      \frac{d_0 - \varrho}{q_3} q_3 = \varrho
    \end{equation*}
    whence $\deg(d_0 - \lambda q_3) = \deg \varrho < \deg Q \le \deg P$.
  \item From \eqref{eq:12} we find
    \begin{equation*}
      \begin{aligned}
      d_1 + \lambda q_2 &= -\frac{2d_0q_2}{2q_3} +
                          \frac{k_5(q_2^2-q_3^2)+k_6(q_0q_3-q_1q_2)}{2q_3m_3} + \frac{d_0-\varrho}{q_3} q_2 \\
          &= \frac{k_5(q_2^2-q_3^2)+k_6(q_0q_3-q_1q_2)}{2q_3m_3} -
            \frac{\varrho q_2}{q_3} \\
          &= \frac{k_5(q_2^2-q_3^2)+k_6(q_0q_3-q_1q_2) - 2\varrho q_2m_3}{2q_3m_3}.
      \end{aligned}
    \end{equation*}
    This is a polynomial whose degree, because of $\deg \varrho < \deg Q$ and
    $\deg K = \deg M + \deg h + \deg \ell$, is at most
    \begin{multline*}
        \max\{\deg K + \deg Q - \deg M, \deg \varrho\} \\
        = \max\{\deg Q + \deg h + \deg \ell, \deg \varrho\}
        = \deg Q + \deg h + \deg \ell = \deg P.
    \end{multline*}
  \item Similar computations using \eqref{eq:13} and \eqref{eq:14} yield
    \begin{equation*}
      d_2 - \lambda q_1 =
          \frac{(k_5(-q_0q_3-q_1q_2)+k_6(q_1^2-q_3^2))}{2q_3m_3} + \frac{\varrho
            q_1}{q_3}
    \end{equation*}
    and
    \begin{equation*}
      d_3 + \lambda q_0 =
    \frac{(k_5(q_1q_3+q_0q_2)+k_6(q_2q_3-q_0q_1))}{2q_3m_3} - \frac{\varrho q_0}{q_3}.
    \end{equation*}
    It can be seen similarly that the degrees of these expressions are bounded
    by $\deg P$.\qedhere
  \end{itemize}
\end{proof}

Combining all results so far, we can state a central result:

\begin{theorem}
  \label{th:1}
  Given a saturated rational kinematic ruled surface $L = \vecLp + \eps
  \vecLd$ with $h = \rgcd(\vecLp)$, there exists a motion polynomial $C = P
  + \eps D$ of degree $\deg C = \frac{1}{2}(\deg L + \deg h)$ such that
  $\ej{C}\qk\cej{C} = hL$. The dual part $D = d_0 + d_1\qi + d_2\qj + d_3\qk$ is
  given by the formulas referenced in Lemma~\ref{lem:5}.
\end{theorem}

In case of $h = 1$, the motion polynomial $C$ is clearly of minimal possible
degree as already the primal part is of minimal degree. We will see later that
minimality is also achieved in case of $\deg h > 0$ but, before doing so, let us
illustrate the computation procedure at hand of an example.

\begin{example}
	\label{ex:1}
  We examine the reduced kinematic line polynomial
  \(K = \vecKp + \eps \vecKd\) where
  \begin{multline*}
    \vecKp=(t^2-6t+10)((2t^3-4t^2+2t-14)\qi\\
    -(2t^3-6t^2+4t+8)\qj+(t^4+t^2+14t-8)\qk),
  \end{multline*}
  and
  \begin{multline*}
    \vecKd=k_5\qi + k_6\qj + k_7\qk \\
    = 112t\qi +(t^6-51t^2+42t-104)\qj +(2t^5-6t^4+2t^3-14t^2-4t+104)\qk.
  \end{multline*}
  The minimal saturating factor is \(\ell = 1\) whence \(L=K\). Moreover, \(h =
  \rgcd(\vecLp) = \rgcd(\vecKp) = t^2-6t+10\).

  In a first step we compute $Q = q_0 + q_1\qi + q_2\qj + q_3\qk \in
  \H[t]$ such that $Q\qk\cj{Q} = M \coloneqq \vecKp/(h\ell) =
  \vecKp/h$:
	\[
    Q=t^2+1+(t-2)\qi+ (t-3)\qj +(t+2) \qk.
  \]
  This already gives us the primal part \(P\) of the sought motion polynomial
  \(C = P + \eps D\) as
  \[
    P = Qh\ell = Qh = (t^2-6t+10)(t^2+1+(t-2)\qi+ (t-3)\qj +(t+2) \qk).
  \]

	To determine the dual part $D=d_0 + d_1\qi + d_2\qj + d_3\qk$, we use the
  extended Euclidean algorithm to compute
  \[
    a = \tfrac{1}{5},\quad
    b = -\tfrac{1}{5}(t-2)
  \]
  such that \(aq_0 + bq_3 = 1\). By Equations~\eqref{eq:12}, \eqref{eq:13},
  \eqref{eq:14}, and \eqref{eq:22} we now have
  \begin{equation*}
    \begin{aligned}
      d_0 &= -\tfrac{1}{10}(t^5-3t^4+t^3-7t^2-2t+52), \\
      d_1 &= \phantom{-}\tfrac{1}{10}(t^5-3t^4+21t^3-27t^2-22t-208), \\
      d_2 &= -\tfrac{1}{10}(t^5-7t^4+21t^3-11t^2+34t-52), \\
      d_3 &= \phantom{-}\tfrac{1}{10}(t^5-3t^4+t^3-7t^2-2t+52)(t-2).
    \end{aligned}
  \end{equation*}
  We see that $\deg D = 6 > 4 = \deg P$. Following Lemma~\ref{lem:5} we use
  polynomial division to compute
  \begin{equation}
    \label{eq:27}
    \lambda = -\tfrac{1}{10}(t^4-5t^3+11t^2-29t+56),\quad \varrho = 6.
  \end{equation}
  such that $d_0 = q_3\lambda + \varrho$. Replacing $D$ with $D + \lambda
  Q\qk$ we find
  \begin{equation*}
    \begin{aligned}
      d_0 \leadsto d_0 - \lambda q_3 &= \phantom{-}6, \\
      d_1 \leadsto d_1 + \lambda q_2 &= \phantom{-}\tfrac{1}{2}(t^4-t^3+7t^2-33t-8), \\
      d_2 \leadsto d_2 - \lambda q_1 &= -2(2t^2-4t+3), \\
      d_3 \leadsto d_3 + \lambda q_0 &= -\tfrac{1}{2}(t^4-5t^3+11t^2-17t+32)
    \end{aligned}
  \end{equation*}
  so that, indeed, $\deg D = 4 = \deg P$ and we obtain the solution
  \begin{multline}
    \label{eq:28}
    C =
    (t^2-6t+10)(t^2+1+(t-2)\qi+ (t-3)\qj +(t+2) \qk)\\
    + \tfrac{1}{2}\eps\bigl(
    12 +(t^4-t^3+7t^2-33t-8)\qi
    -(8t^2-16t+12)\qj
    -(t^4-5t^3+11t^2-17t+32)\qk
    \bigr).
  \end{multline}
  It satisfies $C\cj{C} \in \R[t]$ and $\ej{C}\qk\cj{\ej{C}} = hK = hL$ with $h
  = t^2 - 6t + 10$.

  Note that the low degree coefficients of $\lambda$ are irrelevant and we could
  use any $\lambda = -\frac{1}{10}(t^4-5t^3 + \lambda_2t^2 + \lambda_1t +
  \lambda_0)$ with arbitrary $\lambda_2$, $\lambda_1$, $\lambda_0$ instead of
  \eqref{eq:27} (cf. Remark~\ref{rem:3}).
\end{example}

\section{Solutions of Minimal Degree}
\label{sec:minimal-degree}

The purpose of this section is to show that the solutions to
\eqref{eq:4}--\eqref{eq:6} as described in Theorem~\ref{th:1} are of minimal
degree. We begin with a lemma that shows that no solutions to $\ej{C}\qk\cej{C}
= hL$ exist if $h$ is a proper factor of $\rgcd(\vecLp)$. It uses the notion of
a \emph{reduced} polynomial $L = \vecLp + \eps \vecLd \in \DH[t]$. By this we
mean that $\gcd(\rgcd(\vecLp), \rgcd(\vecLd)) = 1$ that is, there are no
unnecessary real polynomial factors of~$L$.

\begin{lemma}
	\label{lem:6}
	Let $L = \vecLp + \eps \vecLd \in \DH[t]$ be a reduced saturated kinematic
  line polynomial. If $h$ is a proper factor of $\rgcd(\vecLp)$, the system of
  Equations~\eqref{eq:4}--\eqref{eq:6} has no polynomial solution.
\end{lemma}

\begin{proof}
  If $h$ has a linear factor of odd multiplicity, no solution exists. Hence, we
  may assume that all linear factors of \(h\) are of even multiplicity. Define
  $g \coloneqq \rgcd(\vecLp)$ and assume that $f \coloneqq g / h$ is a nonzero
  polynomial of positive degree. Then the linear factors of \(f\) are of even
  multiplicity as well.

  Proceeding similar to the proof of Lemma~\ref{lem:3} (and of
  \cite[Lemma~5]{derin25}) we arrive at
  \begin{equation}
    \label{eq:29}
    d_3 = -\frac{2fd_0q_0 + \frac{1}{2}k_7}{2q_3f}.
  \end{equation}
  Compare this with \eqref{eq:21} and note the additional factor \(f\) in one
  summand of the numerator and in the denominator. We claim that there is no
  \(d_0 \in \R[t]\) such that the right-hand side of Equation~\eqref{eq:29} is
  polynomial. A necessary condition for this is
  \begin{equation}
    \label{eq:30}
    4fd_0q_0 = k_7 \mod q_3f
  \end{equation}
  Since \(\gcd(\rgcd(\vecKd), f) = 1\) and \(\rgcd Q = 1\), our usual argument
  (multiplying \(L\) from the left with a suitable unit quaternion) allows to
  assume, without loss of generality, that \(\gcd(k_7,f) = \gcd(k_7,q_3) = 1\).
  But then, the left-hand side of \eqref{eq:30} is not invertible module
  \(q_3f\) while the right-hand side is. Therefore, no solution exists.
\end{proof}

The next series of lemmas is similar to lemmas proved in \cite{derin25}.
Starting with a motion polynomial $C$ satisfying $\ej{C} \qk \cej{C} = hL$ and
assuming existence of a factor $f$ of $h / \gcd(\rgcd(\vecLp),h)$ we show that
there exists a right factor $E$ of $C$ that fixes $\qk$. Thus, $E$ can be
divided off from $C$ and a solution of lower degree is obtained.

\begin{lemma}
	\label{lem:7}
	Suppose the Equations~\eqref{eq:4}--\eqref{eq:6} have a motion polynomial
  solution \(C = P + \eps D\) and $f$ is an irreducible real quadratic factor of
  $h/\gcd(\rgcd(\vecLp),h)$ whose multiplicity as factor of $h$ is one. Then the
  system of equations also has a solution if $h$ is replaced by $\frac{h}{f}$.
\end{lemma}

\begin{proof}
  Our aim is to construct a new motion polynomial $\tilde C$ such that
  $C=\tilde{C}E$ for some motion polynomial $E$ satisfying $\ej{E} \qk
  \cj{\ej{E}} = f\qk$. Then
  \[
    \ej{\tilde{C}} \qk \cej{\tilde{C}} =
    \tfrac{1}{f}\ej{(\tilde{C}E)} \qk \cej{(\tilde{C}E)} =
    \tfrac{1}{f}\ej{C} \qk \cj{\ej{C}} =
    \tfrac{h}{f}L.
  \]
  There exist \(F = f_0 + f_3\qk \in \C[t] \setminus \R[t]\) and \(Q \in \H[t]\)
  such that \(P = QF\) and \(F\cj{F} = f\). Moreover, \(f\) is not a factor of
  \(Q\) by our assumptions. Then, \cite[Lemma~2 and Lemma~3]{hegedus13}
  guarantee existence of motion polynomials $\tilde{C}$, $E \in \DH[t]$ such
  that $C = \tilde{C}E$ and $E\cj{E} = f$. This implies $\ej{E} \qk \cej{E} =
  f\qk$.
\end{proof}

\begin{example}
  \label{ex:2}
	We revisit the line polynomial $L$ from Example~\ref{ex:1} and examine the
  motion polynomial
	\begin{multline*}
		\hat{C} =
    - 2t^3 + 13t^2 - 26t + 10 + \qi(t^4 - 8t^3 + 20t^2 - 8t - 20) \\
    +\qj(-t^4 + 9t^3 - 31t^2 + 48t - 30) + \qk(t^5 - 6t^4 + 12t^3 - 10t^2 + 8t + 20)\\
    +\tfrac{1}{5}\eps(
    14t^5-37t^4+52t^3-41t^2+29t+108 + \qi(14t^4 + 81t^3 - 13t^2 - 183t + 92)\\
    + \qj(14t^4 + 140t^3 + 230t^2 + 21t - 48) + \qk(28t^3 - 43t^2 + 140t - 34)).
	\end{multline*}
	This motion polynomial satisfies the identity:
	\[
    \ej{\hat{C}} \qk \cej{\hat{C}} = h L
	\]
	where
	\[
	h = (t^2 + 1)g \quad\text{and}\quad g = \rgcd(L_p) = t^2-6t+10.
	\]
	According to Lemma~\ref{lem:5}, there exists a linear motion polynomial $E =
  e_0 + e_3\qk + \eps(e_4 + e_7\qk) \in \DH[t]$ that satisfies \(E\cj{E} =
  t^2+1\) and is a right factor of $\hat{C}$. Indeed, applying the factorization
  algorithm of \cite{hegedus13} yields
	\[
    E = 1 + t\qk + \eps(t - \qk).
  \]
	It satisfies $\hat{C} = C E$, where $C$ is the motion polynomial defined in
  Equation~\eqref{eq:28}.
\end{example}

\begin{lemma}
	\label{lem:8}
	If the system of equations \eqref{eq:4}--\eqref{eq:6} has a motion polynomial
  solution and if there exists a real polynomial \(f \in \R[t]\), either linear
  or quadratic and irreducible over \(\mathbb{R}\), such that \(f^2\) divides \(
  h/\gcd(\rgcd(\vecLp), h)\), then there exists an integer \(m \geq 1\) and a
  solution to the same system \eqref{eq:4}--\eqref{eq:6} with \(h\) replaced by
  \(h/f^{2m}\).
\end{lemma}

Lemma~\ref{lem:8} is very similar to \cite[Lemma~7]{derin25} and so is its
proof. Therefore, we will sometimes refer the reader to~\cite{derin25}.

\begin{proof}[Proof of Lemma~\ref{lem:8}]
	Assume we already have a solution to the system equations
  \eqref{eq:4}--\eqref{eq:6} represented as the motion polynomial \( C = P +
  \eps D \). The primal part can be expressed in the form \( P = Q f^m \) where
  $ m \in \mathbb{N} $ and \( Q \in \H[t] \) such that \( \gcd(\rgcd(Q), f) = 1 \).

	We aim to construct a motion polynomial \(E = f^m + \eps F \in \DH[t]\) where
  \(F = e_7\qk \in \H[t]\) is a right factor of $C$. The construction of \(F\)
  is literally the same as in the proof of \cite[Lemma~7]{derin25}. We obtain
  the factorization
  \[
    C = f^mQ + D = (Q + \eps K)(f^m + \eps F)
  \]
  for suitable \(K\), \(F \in \H[t]\) such that \(\deg F < f^m\).

  Because the norm of \(C\) equals \(f^{2m}Q\cj{Q}\), the norm of \(f^m + \eps
  F\) equals \(f^{2m}\) and \(f^m + \eps F\) is a motion polynomial whence \(F =
  e_5\qi + e_6\qj + e_7\qk\) for some real polynomials \(e_5\), \(e_6\), \(e_7
  \in \R[t]\). We still need to show that \(e_5 = e_6 = 0\). To do this, we
  recall that \(f^{2m}\) is a factor of \(\ej{C}\qk\cej{C}\). A straightforward
  computation gives
  \[
    \ej{C}\qk\cej{C} =
    f^{2m}(Q-\eps K)\qk(\cj{Q}-\eps\cj{K}) + f^m\eps Q(\qk F - F\qk)\cj{Q}.
  \]
  Thus, \(f^m\) needs to be a factor of \(Q(\qk F - F\qk)\cj{Q} = 2(qe_5\qj -
  qe_6\qi)\) where \(q = Q\cj{Q}\). Since \(f\) is not a factor of \(q\), we
  infer that \(f^m\) divides both, \(e_5\qj\) and \(e_6\qi\). But since \(\deg
  e_5 \le \deg F < \deg f^m\) we have \(e_5 = 0\) and similarly \(e_6 = 0\) as
  well.
\end{proof}

\begin{example}
	We consider again the line polynomial $L$ of Examples~\ref{ex:1} and
  \ref{ex:2} and the motion polynomial
  \begin{multline*}
    \tilde{C} =
      t^6-6t^5+12t^4-12t^3+21t^2-6t+10
      +\qi(t^5-8t^4+23t^3-28t^2+22t-20)\\
      +\qj(t^5-9t^4+29t^3-39t^2+28t-30)
      +\qk(t^5-4t^4-t^3+16t^2-2t+20)\\
      + \tfrac{1}{2}\eps (
          -2t^3+20t^2+4t-28
          +(t^6-t^5+8t^4-32t^3-19t^2+23t-68)\qi\\
          +(-8t^4+14t^3-4t^2-28t+28)\qj
          +(-t^6+5t^5-10t^4+10t^3-21t^2+5t-12)\qk)
  \end{multline*}
  It satisfies $\tilde{C}_{\eps} \qk \cej{\tilde{C}} = hL$ where $h =
  (t^2+1)^2g$ and $g = \rgcd(\vecLp) = t^2-6t+10$. Thus, the assumptions of
  Lemma~\ref{lem:8} are fulfilled with \(f = t^2+1\). By the proof of this
  lemma, there exists a quadratic motion polynomial $E = f + \eps e_7\qk \in
  \DH[t]$ that is a right factor of $\tilde{C}$. Indeed, with $E = t^2+1 +
  \eps\qk$ we have $\tilde{C} = CE$ where $C$ is the original motion polynomial
  of Example~\ref{ex:2}.
\end{example}

Combining the results of Lemmas~\ref{lem:6}--\ref{lem:8} we can now argue as follows:
\begin{itemize}
\item If $h$ contains a polynomial factor that is not a factor of
  $\rgcd(\vecLp)$ we can, by Lemma~\ref{lem:7} and Lemma~\ref{lem:8}, reduce the
  degree of a solution polynomial~$C$.
\item Iterating this process, Lemma~\ref{lem:6} ensures that we will ultimately
  end with $h = \rgcd(\vecLp)$.
\end{itemize}
This proves

\begin{theorem}
  \label{th:2}
  Given a saturated kinematic rational ruled surface $L = \vecLp +
  \eps\vecLd$ with $\rgcd(\vecLp) = h$, the motion polynomials of minimal
  degree are precisely those described in Theorem~\ref{th:1}. The minimal degree
  motion polynomial is unique if and only if $\rgcd(\vecLp) = 1$.
\end{theorem}

\section{The Minimal Degree Motion of the Cylindroid}
\label{sec:cylindroid}

Now we apply our theory to the cylindroid (Plücker conoid), a famous rational
ruled surface with important applications in kinematics and mechanism science.
It is also considered in \cite{selig24} where the authors, based on the cubic
rational parametrization
\begin{equation}
  \label{eq:31}
  L = t(1+t^2)\qi + (1+t^2)\qj - 2t\eps(\qi - t\qj)
\end{equation}
of the cylindroid, construct a rational motion to guide a line along its
rulings.

We attempt to compute rational motions to generate $L$ which are of minimal
degree. However, the norm polynomial $L\cj{L} = (t^2+1)^3$ is not a square and a
rational motion creating exactly the parametrization \eqref{eq:31} cannot exist.
This seeming contradiction to results of \cite{selig24} can be explained by the
observation that the surface's spherical image, as induced by \eqref{eq:31}, is
an irrationally parametrized great circle on the unit sphere. Thus,
\eqref{eq:31} is not kinematic. However, a circle can be parametrized
rationally. Using the rational two-to-one parameterization
\begin{equation}
  \label{eq:32}
  L = (1+t^2)^2((1-t^2)\qi + t\qj) - 4t(1-t^2)\eps(2t\qi - (1-t^2)\qj)
\end{equation}
instead of \eqref{eq:31} we see that, indeed, $L\cj{L} = (1+t^2)^6$ is a square
and we expect $L$ to appear as line trajectory of a rational motion.

Denote by \(L_p\) the primal part of \(L\). We have \(g = \rgcd(L_p) =
(1+t^2)^2\) and \(M = L_p/g = (1-t^2)\qi + 2t\qj\) and encounter another
obstacle, namely \(m_3 = 0\) which contradicts our genericity assumptions and
makes a lot of essential equations, most importantly
\eqref{eq:21}--\eqref{eq:24}, invalid. Therefore, we base our computation not on
the parametric equation \eqref{eq:32} of the cylindroid but on
\begin{equation*}
  L = (1+t^2)^2((1-t^2)\qk + t\qj) - 4t(1-t^2)\eps(2t\qk - (1-t^2)\qj)
\end{equation*}
where $\qi$ and $\qk$ have been interchanged.\footnote{An alternative would be
  to derive different versions of \eqref{eq:21}--\eqref{eq:24} under the
  assumption that either $m_1 \neq 0$ or $m_2 \neq 0$.} This requires to updated
$L_p$ and $M$ accordingly but does not change $g$.

Using the procedure proposed in \cite{schroecker24} we compute \(Q = -t\qj -
\qk\) such that \(M = Q \qk \cj{Q}\).

Theorem~\ref{th:1} predicts existence of a motion polynomial \(C = P + \eps D \in
\DH[t]\) such that \(\ej{C} \qk \cej{C} = gL\) and \(P = Qg\). The dual part
\(D = d_0 + d_1\qi + d_2\qj + d_3\qk\) is found from Equation~\eqref{eq:24} and
Equations~\eqref{eq:21}--\eqref{eq:24}.

Because of \(q_0 = 0\) and \(q_3 = -1\), Equation~\eqref{eq:24} is satisfied
with \(a = 0\) and \(b = -1\). Thus, \(d_0 = 0\) by \eqref{eq:24}. Plugging this
into \eqref{eq:21}--\eqref{eq:24}, we find
\[
  d_1 = 0,\quad
  d_2 = -2t(t^2-1),\quad
  d_3 = 2t^2(t^2-1).
\]
These expressions are polynomial, as expected, and already of degree less than
\(\deg P = 5\). Thus, the polynomial
\[
  C = P + \eps D = (t^2+1)^2(-t\qj - \qk) + 2t(t^2-1)\eps(t\qk-\qj)
\]
is a motion polynomial of minimal degree and satisfies \(\ej{C}\qk\cej{C} =
gL\). This solution is not unique. All further minimal degree solutions can be
written as
\begin{equation}
  \label{eq:33}
  C + d\eps(1-\qi)
\end{equation}
where \(d \in \R[t]\) is of degree at most five.

Up to coordinate changes and re-parametrization, the motion given in
\cite[Equation~(6)]{selig24}, which was derived in a completely different way,
is one instance of~\eqref{eq:33}. A mechanism to draw the cylindroid based on
this parametrization is depicted in \cite[Figure~7]{selig24}.

\section{Mechanism Design}
\label{sec:mechanism-design}

In this section we illustrate how to construct a mechanism with a prescribed
kinematic rational ruled surface $L$ as trajectory of a moving line. The idea is
to first compute the minimal degree motion $C$ and then use motion factorization
\cite{hegedus13,li18} to construct the mechanism. Motion factorization is
meanwhile a mature technology with a free and easy to use reference
implementation \cite{huczala2024}. Therefore, we confine ourselves to one
example.

In fact, we do a bit more than outlined above. Our input will not be the
rational kinematic ruled surface $L$ but three lines
\begin{equation}
  \label{eq:34}
  \begin{aligned}
    L_0 = L_{0,p} + \eps L_{0,d} =& -0.5692893982 \qi + 0.1879452202 \qj + 0.8003662757 \qk \\
         & + \eps(2.394938124 \qi - 1.698415049 \qj + 2.102314808 \qk),\\
    L_1  = L_{1,p} + \eps L_{1,d}=& -0.9543899389 \qi + 0.1373020929 \qj + 0.2651188032 \qk \\
         & + \eps(0.8624265591 \qi + 1.131167268 \qj + 2.518793799 \qk),\\
    L_2  = L_{2,p} + \eps L_{2,d}=& -0.9834592139 \qi - 0.04634208347 \qj - 0.1751010735 \qk  \\
         & + \eps(-0.4177553080\qi + 2.772307919 \qj + 1.612615937\qk)
  \end{aligned}
\end{equation}
that are to be interpolated by $L$ in that order. As a consequence, also the
mechanism we design will guide a moving line such that it coincides with $L_0$,
$L_1$, and $L_2$. The interpolation conditions are
\begin{equation}
  \label{eq:35}
  L(t_0) = w_0 L_0,\quad
  L(t_1) = w_1 L_1,\quad
  L(t_2) = w_2 L_2
\end{equation}
where $t_0 < t_1< t_2 \in \R$ are the parameter times of the respective lines
and $w_0$, $w_1$, $w_2 \in \R \setminus \{0\}$ are some weights that can be
chosen. We pick $t_0 = -1$, $t_1 = 0$, $t_2 = 1$ and $w_0 = w_1 = w_2 = 1$.

Since $L$ is supposed to be kinematic, its primal parts should be of the shape
$\vecLp = P\qk\cj{P}$ for some polynomial $P \in \H[t]$. Denoting by
\begin{equation*}
  \mathcal{L}_0 = \frac{t(t-1)}{2},\quad
  \mathcal{L}_1 = -\frac{t-1}{t+1},\quad
  \mathcal{L}_2 = \frac{t(t+1)}{2}
\end{equation*}
the Lagrange polynomials with knot vector $(t_0,t_1,t_2)$, we make the ansatz $P
= p_0\mathcal{L}_0 + p_1\mathcal{L}_1 + p_2\mathcal{L}_2$. By
Equation~\eqref{eq:35}, the yet to be determined coefficients $p_0$, $p_1$, $p_2
\in \H$ are subject to
\begin{equation}
  \label{eq:36}
  p_i\qk\cj{p}_i = L_{p,i},\quad i \in \{0,1,2\}.
\end{equation}

Conditions of this type are well-known, for example in the context of the
description of Pythagorean hodograph curves via their quaternionic pre-image
\cite[Chapter~22]{farouki08}. By a minor variation of
\cite[Equation~2.13]{sir07}, all solutions to \eqref{eq:36} are given as
\begin{equation}
  \label{eq:37}
  p_i =
  \sqrt{\vert L_{p,i} \vert}
  \frac{\frac{L_{p,i}}{\vert L_{p,i}\vert} + \qk }{\bigl\vert\frac{L_{p,i}}{\vert L_{p,i}\vert} + \qk\bigr\vert }(\cos\varphi_i + \sin\varphi_i \qk)
\end{equation}
where $\vert L_{p,i} \vert \coloneqq \sqrt{L_{p,i}\cj{L}_{p,i}}$ and $\varphi_i
\in [0,2\pi)$.\footnote{The intuition behind \eqref{eq:37} is as follows: For
  $\varphi = 0$, $p_i$ just describes the half-turn around the bisector of $\qk$
  and $L_{p,i}$, composed with a suitable scaling. All rotations and scalings to
  map $\qk$ to $L_{p,i}$ are obtained by right-multiplying with a rotation that
  fixes $\qk$.} We pick the solutions corresponding to $\varphi_0 = \varphi_1 =
\varphi_2 = 0$:
\begin{equation*}
  \begin{aligned}
    p_0 = -0.3000113351\qi + 0.09904575183\qj + 0.9487798153\qk,\\
    p_1 = -0.5999916401\qi + 0.08631703309\qj + 0.7953360307\qk,\\
    p_2 = -0.7656688635\qi - 0.03607947323\qj + 0.6422222851\qk.
  \end{aligned}
\end{equation*}
The are vectorial and determine the primal part
\begin{equation*}
  \begin{aligned}
  \vecLp &= (p_0\mathcal{L}_0 + p_1\mathcal{L}_1 + p_2\mathcal{L}_2)\qk
          \cj{(p_0\mathcal{L}_0 + p_1\mathcal{L}_1 + p_2\mathcal{L}_2)} \\
            &= (0.000022162 \qi-0.0000180973 \qj-0.0075160582 \qk) t^4 \\
              &\quad-(0.0206626527 \qi-0.0167874447 \qj-0.0238095903 \qk) t^3 \\
              &\quad+(0.177993471 \qi-0.0664824272 \qj+0.0550298561 \qk) t^2 \\
              &\quad-(0.1864222551 \qi+0.1339310966 \qj+0.5115432649 \qk) t \\
              &\quad-0.9543899389 \qi+0.1373020929 \qj+0.2651188032 \qk
  \end{aligned}
\end{equation*}

Now, the dual part $\vecLd$ of $L$ is subject to Plücker condition
\begin{equation}
  \label{eq:38}
  \vecLp\cjvecLd + \vecLd\cj\vecLp = 0
\end{equation}
and the interpolation constraints
\begin{equation}
  \label{eq:39}
  \vecLd(t_i) = w_iL_{i,d},\quad i \in \{0,1,2\}.
\end{equation}
Equations~\eqref{eq:38} and \eqref{eq:39} give rise to a system of linear
equations to determine the unknown quaternionic coefficients $\ell_{d,i} \in \H$
of $\vecLd = \sum_{i=0}^4 \ell_{d,i}t^i$. We obtain the unique solution
\begin{equation*}
  \begin{aligned}
  \vecLd &= (0.000022162\qi - 0.0000180973\qj - 0.0075160582\qk) t^4 \\
               &\quad- (0.0206626527\qi - 0.0167874447\qj - 0.0238095903\qk) t^3 \\
               &\quad+ (0.177993471\qi - 0.0664824272\qj + 0.0550298561\qk) t^2 \\
               &\quad- (0.1864222551\qi + 0.1339310966\qj + 0.5115432649\qk) t \\
               &\quad- 0.9543899389\qi + 0.1373020929\qj + 0.2651188032\qk.
  \end{aligned}
\end{equation*}

The sought motion polynomial $C = P + \eps D$ is now determined by the already
computed primal part $P$ and the dual part
\begin{equation*}
  \begin{aligned}
  D &= (-0.4263398024\qi - 0.521506355\qj + 0.2008885718\qk - 0.9525313145) t^2 \\
    &\quad- (0.4807564566\qi - 1.302422976\qj + 0.0368980596\qk - 2.837641695) t \\
    &\quad- 10.49908633 + 7.295181813\qj - 0.7917388696\qk
  \end{aligned}
\end{equation*}
which we can compute by using the formulas derived previously.

The motion polynomial $C$ is of degree two and therefore describes the coupler
motion of a Bennett linkage \cite{brunnthaler05,hamann11}. Its axes can be found
from the two factorizations of $C = c_2(t-h_1)(t-h_2) = c_2(t-k_1)(t-k_2)$ into linear factors:
\begin{equation}
  \label{eq:40}
  \begin{aligned}
    c_2 &= 0.0671515410\qi - 0.05483389380\qj + 0.0001650193\qk \\
        &\quad-\eps(0.9525313145 + 0.4263398024\qi + 0.5215063550\qj - 0.2008885718\qk), \\
    h_1 &= 3.437729498 +0.4245869672\qi - 0.6826178148\qj - 1.925681116\qk\\
        &\qquad-\eps(39.06204587\qi + 23.42489306\qj + 0.3089744570\qk),\\
    h_2 &= -1.847095642 + 0.6951435567\qi + 2.046951289\qj - 0.3765517306\qk \\
        &\qquad+ \eps(47.58933550\qi - 20.43787422\qj - 23.24757084\qk),\\
    k_1 &= -1.847095642 + 1.867767136\qi + 1.089422750\qj - 0.3736701032\qk\\
        &\qquad+\eps(23.23952734\qi - 50.18964901\qj - 30.16489659\qk),\\
    k_2 &= 3.437729498 - 0.7480366121\qi + 0.2749107239\qj - 1.928562744\qk\\
        &\qquad-\eps(14.71223772\qi - 6.326881730\qj - 6.608351290\qk).
  \end{aligned}
\end{equation}
The (non-normalized) Plücker coordinates of the four revolute axes in a certain
configuration of the mechanism and transformed by $c_2^{-1}$ are the
$\eps$-conjugates of the factors' vector parts:
\begin{equation*}
  (h_1 - \cj{h}_1)_\eps,\quad
  (h_2 - \cj{h}_2)_\eps,\quad
  (k_2 - \cj{k}_2)_\eps,\quad
  (k_1 - \cj{k}_1)_\eps.
\end{equation*}

For visualizing the linkage (cf. Figure~\ref{fig:1}) and animating it, we used
the open source Python package \texttt{rational\textunderscore linkages}
described in \cite{huczala2024}. It can also be used for computing the
factorizations \eqref{eq:40} of the motion polynomial~$C$.

\begin{figure}
  \centering
\includegraphics[page=1]{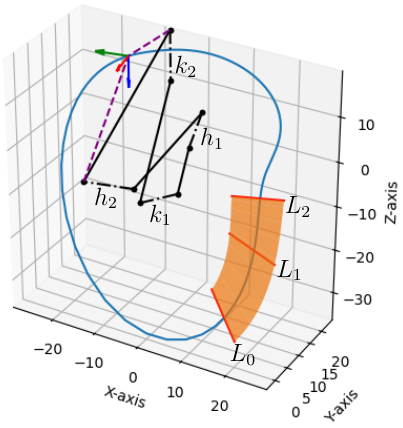}\includegraphics[page=2]{figures}\\
  \includegraphics[page=3]{figures}\includegraphics[page=4]{figures}
  \caption{Bennett linkage with revolute axes in dash-dotted linestyle. In the
    top-left picture, they are labeled by their corresponding factors in the
    factorization \eqref{eq:40}. Links are drawn as solid lines. The moving
    frame is attached to the link connecting $h_2$ and $k_2$. Its $z$-axis
    interpolates four prescribed lines $L_0$, $L_1$, and $L_2$ at parameter
    times $t = -1$ (top right), $t = 0$ (bottom left), and $t = 1$ (bottom
    right). The trajectory of the origin and a relevant part of the kinematic
    ruled surface through these lines is displayed as well.}
  \label{fig:1}
\end{figure}

\section{Conclusion}
\label{sec:conclusion}

This text presents a complete solution to the problem of creating a rational
ruled surface $L$ as line trajectory of a rational motion $C = P + \eps D$ of
minimal degree in the dual quaternion model of space kinematics. In contrast to
point trajectories, the resulting motion does not generally exist and, if it
exists, need not be unique. Existence is related to $L$ being ``kinematic'',
uniqueness of the minimal degree motion (up to translation along and rotation
around the moving line) is related to non-trivial real polynomial factors of the
primal part of $L$. The primal $P$ of the rational motion can be computed by
methods known from literature. Our approach provides more or less explicit
formulas for the dual part $D$. Besides elementary arithmetic operations, only
polynomial division with remainder and the extended Euclidean algorithm are
required.

From the resulting rational motion, closed-loop mechanisms \cite{hegedus13} or
networks of such mechanisms \cite{li18} can be created via motion factorization,
at least in generic cases. Software implementations for this task do exist
\cite{huczala2024} and a ``Universality Theorem for Kinematic Ruled Surfaces''
along the lines of \cite{Gallet16,li18} seems to be within reach.

In its original version, Kempe's Universality Theorem speaks about the creation
of \emph{algebraic} curves as point trajectories of linkages. In the light of
this, the construction of minimal degree \emph{algebraic} motions to a given
algebraic ruled surface is an interesting topic of research. This problem has
not yet been addressed at all, also not for point and plane trajectories.

A gap in the synthesis procedure along these lines is interpolation of suitable
ruled surfaces. While some research has been dedicated to interpolation
\cite{odehnal08,odehnal17} and approximation \cite[Chapter~4]{pottmann10} with
ruled surfaces, interpolation or approximation procedures with the kinematicity
of the ruled surface as additional constraint have not been considered so far.
We presented one example in Section~\ref{sec:mechanism-design} but fundamental
issues still remain to be explored.

\section*{Acknowledgment}

Zülal Derin Yaqub was supported by the Austrian Science Fund (FWF) P~33397-N
(Rotor Polynomials: Algebra and Geometry of Conformal Motions) and also
gratefully acknowledges the support provided by the Scientific and Technological
Research Council of Türkiye, TUBITAK-2219 -- International Postdoctoral Research
Fellowship Program for Turkish Citizens.

\bibliographystyle{elsarticle-num}

\end{document}